\documentclass[review,12pt]{elsarticle}

\usepackage{amsmath, amsthm, amssymb, amsfonts}
\usepackage{geometry}
\usepackage[colorlinks=true, allcolors=blue]{hyperref}

\newtheorem{theorem}{Theorem}[section]
\newtheorem{lemma}[theorem]{Lemma}

\theoremstyle{definition}
\newtheorem{definition}[theorem]{Definition}

\newtheorem{remark}[theorem]{Remark}

\numberwithin{equation}{section}

\newcommand{\ep}{\varepsilon}
\def\R{{\mathbb{R}}}
\def\d{\displaystyle}
\def\e{{\varepsilon}}

\def\U{{\mathcal{F}^{\eta}_d}}
\def\V{{\mathcal{G}^{\eta}_d}}
\def\w{{\mathcal{\psi}^{\eta}_d}}
\def\GG{\V}

\begin{document}

\begin{frontmatter}

\title{
On the blow-up of solutions to scale-invariant wave equations with damping and mass: Beyond the positive discriminant restriction}
\author{Mohamed Ali Hamza}
\ead{mahamza@iau.edu.sa}
\address{Department of Basic Sciences, Deanship of Preparatory Year and Supporting Studies, Imam Abdulrahman Bin Faisal University, P. O. Box 1982, Dammam, Saudi Arabia.}

\begin{abstract}
In this article, we study the blow-up of solutions to scale-invariant semilinear wave equations with damping, mass, and a time-derivative nonlinearity:
$$\partial^2_{t}u-\Delta u+\frac{\mu}{1+t}\partial_tu+\frac{\nu^2}{(1+t)^2}u=|\partial_tu|^p, \qquad (x,t) \in \mathbb{R}^n\times[0,\infty),$$
with small initial data. In the existing literature, the analysis of such models is  exclusively restricted to the case where the discriminant $\delta = (\mu - 1)^2 - 4\nu^2$ is non-negative. This restriction is  considered essential because $\delta \ge 0$ ensures the non-oscillatory behavior of the underlying linear solution, which is fundamental for the application of standard test function methods. The principal contribution of this work is the extension of the blow-up result to the regime $\delta < 0$. We demonstrate that the blow-up region, defined by the Glassey-type critical exponent $p \in (1, 1 + \frac{2}{n+\mu-1})$, remains invariant regardless of the sign of the discriminant. Consequently, we clarify that the restriction $\delta \ge 0$ previously considered essential is necessitated by the limitations of standard test function methods rather than the intrinsic nature of the problem.
\end{abstract}

\end{frontmatter}

\section{Introduction}
We study  the following scale-invariant semilinear wave equation featuring damping, mass, and a time-derivative nonlinearity:
\begin{equation}\label{Equ}
\begin{cases}
\partial^2_{t}u-\Delta u+\frac{\mu}{1+t}\partial_tu+\frac{\nu^2}{(1+t)^2}u=|\partial_tu|^p, & (x,t) \in \mathbb{R}^n\times[0,\infty), \\
(u(x,0), \partial_tu(x,0))= \varepsilon(f(x), g(x)), & x\in \mathbb{R}^n,
\end{cases}
\end{equation}
where $\mu \ge 0$, $\nu^2 \ge 0$, and $p > 1$. The initial data $(f, g)$ are assumed to be positive functions with compact support in $B(0,R)$, where 
 $\varepsilon > 0$ represents the data size.

\par

In the absence of damping and mass ($\mu=\nu=0$), problem \eqref{Equ} reduces to the classic semilinear wave equation with derivative nonlinearity. The blow-up phenomenon in this context is well-documented and closely tied to the Glassey conjecture. The critical threshold is determined by the power $p_{\mathrm{Gla}}(n)$, defined as
\begin{equation}\label{Glassey}
p_{\mathrm{Gla}}(n) := 1+\frac{2}{n-1}.
\end{equation}
It is well-established that for $1 < p \le p_{\mathrm{Gla}}(n)$, no global solution exists under suitable positivity assumptions on the data, whereas for $p > p_{\mathrm{Gla}}(n)$, a global-in-time solution exists for sufficiently small $\varepsilon$ (see, e.g., \cite{Hidano1,Hidano2,John1,Sideris,Tzvetkov,Zhou1}).

\par

When damping is introduced ($\mu > 0, \nu = 0$), the blow-up region $p \in (1, p_{\mathrm{Gla}}(n+\mu)]$ and the associated lifespan estimates were established in \cite{Our2}, thereby improving upon the thresholds previously reported in \cite{PT, LT2}. Specifically, for the one-dimensional case, it was recently established \cite{HFglobal} that $p_{\mathrm{Gla}}(1+\mu)$ acts as the critical threshold for $\mu \in (0,2]$. These results confirm that the damping effectively modifies the dynamics, shifting the effective dimension of the problem to $n+\mu$.

\par

For non-vanishing $\mu$ and $\nu$, recent studies have extensively explored the semilinear wave equation \eqref{Equ}. Naturally, the global dynamics and blow-up phenomena of this nonlinear model are fundamentally determined by the properties of the associated linear operator, namely
\begin{equation}\label{Equl}
L_{\mu, \nu} (u):= \partial^2_{t}u-\Delta u+\frac{\mu}{1+t}\partial_tu+\frac{\nu^2}{(1+t)^2}u.
\end{equation}
A crucial element in this analysis is the discriminant
\begin{equation}\label{delta}
\delta = \delta(\mu, \nu) := (\mu-1)^2 - 4\nu^2,
\end{equation}
which characterizes the interplay between damping and mass. Qualitatively, the case $\delta \ge 0$ suggests that the dissipation term $\frac{\mu}{1+t}\partial_t u$ dominates the mass term $\frac{\nu^2}{(1+t)^2}u$.
 Crucially, $\delta$ is invariant under the transformation $V(t,x) = (1+t)^\alpha U(t,x)$. 
 Specifically, if $U$ satisfies  $L_{\mu, \nu} (U)=0$, then $V$ satisfies a similar equation with modified coefficients $\mu(V) = \mu(U)-2\alpha$ and $\nu^2(V) = \alpha^2 - (\mu-1)\alpha + \nu^2$, leaving the discriminant $(\mu(V)-1)^2 - 4\nu^2(V) = \delta(U)$ unchanged.

\par

To date, the existing literature on the semilinear wave equation with scale-invariant damping and mass has focused  exclusively on the regime $\delta \ge 0$. This includes extensive studies on power-type nonlinearities $|u|^q$, as well as very recent investigations into derivative-type $ |\partial_t u|^p$ and mixed nonlinearities $f(u, \partial_t u) = |u|^q + |\partial_t u|^p$ (see, e.g., \cite{Our3, PT, PalmieriNonlinear2018, PalRei18, DabbPal18, Pal19RF, Palmieri-phd} and references therein). This condition is typically invoked to ensure the existence of positive test functions. However, as noted in \cite{PT}, the assumption $\delta \ge 0$ appears to be a technical artifact of the integral transform method required to maintain real-valued indices rather than an intrinsic physical requirement of the blow-up mechanism.

\par

For the regime $\delta \ge 0$, a blow-up result for \eqref{Equ} was established in \cite{Our3}, providing an improvement over the results in \cite{PT}. In the latter, the blow-up region was restricted to $p \in (1, p_{\mathrm{Gla}}(n+\sigma)]$ 
with a shift parameter $\sigma$ depending on $\sqrt{\delta}$ 
for $\delta \in (0,1)$, but through a more refined functional approach, the analysis in \cite{Our3} extended this range to $p \in (1, p_{\mathrm{Gla}}(n+\mu)]$ for all non-negative $\delta$. This outcome confirms that the blow-up dynamics recover the characteristic behavior of the massless case whenever the discriminant is non-negative.\\

\par

In the present work, we show that  the positivity of the discriminant 
 is due to  a technical tool of previous methods rather than an intrinsic property of the equations. More precisely, we extend the blow-up result to the regime $\delta < 0$ for solutions of \eqref{Equ}, establishing that the subcritical region $p \in (1, p_{\mathrm{Gla}}(n+\mu))$ remains unchanged regardless of the sign of $\delta$. 
Ultimately, our framework generalizes and simplifies the approach initiated in \cite{Our3}, while maintaining broad applicability to a wider class of problems, including coupled systems and Tricomi-type models

Throughout this paper, $C$ denotes a generic positive constant which may depend on the parameters $n, p, \mu, d, \nu, \eta, R, f,$ and $g$, 
but is independent of $\varepsilon$. The value of $C$ may change from line to line. Where necessary, the specific dependence of a constant on certain parameters will be explicitly indicated (e.g., $C_{\eta, R}$).

\par

The outline of this article is organized as follows. In Section \ref{sec-main}, we define the notion of solutions for \eqref{Equ} in the energy space and state our main result. Section \ref{sec3} is devoted to the derivation of several technical lemmas essential for the analysis. Finally, the proof of the main theorem is provided in Section \ref{sec44}.

\section{Main Result}\label{sec-main}
In this section, we present our main result concerning the blow-up of solutions to \eqref{Equ}. We begin by introducing the notion of a weak solution within the natural energy space associated with the problem. Specifically, the weak formulation of \eqref{Equ} is defined as follows:

\begin{definition}\label{def1}
Let $T > 0$. We say that $u$ is an energy solution of \eqref{Equ} on $[0,T)$ if 
\begin{equation*}
u \in \mathcal{C}([0,T), H^1(\mathbb{R}^n)) \cap \mathcal{C}^1([0,T), L^2(\mathbb{R}^n)),
\end{equation*}
with  $\partial_t u \in L^p_{loc}(\mathbb{R}^n \times (0,T))$, such that for any test function $\Phi \in \mathcal{C}_0^{\infty}(\mathbb{R}^n \times [0,T))$ and for all $t \in [0,T)$, the following identity holds:
\begin{equation} \label{energysol2}
\begin{aligned}
&\int_{\mathbb{R}^n} \partial_t u(x,t) \Phi(x,t) \, dx - \int_{\mathbb{R}^n} \partial_t u(x,0) \Phi(x,0) \, dx - \int_0^t \int_{\mathbb{R}^n} \partial_t u(x,s) \partial_t \Phi(x,s) \, dx \, ds \\
&\quad + \int_0^t \int_{\mathbb{R}^n} \nabla u(x,s) \cdot \nabla \Phi(x,s) \, dx \, ds + \int_0^t \int_{\mathbb{R}^n} \frac{\mu}{1+s} \partial_t u(x,s) \Phi(x,s) \, dx \, ds \\
&\quad + \int_0^t \int_{\mathbb{R}^n} \frac{\nu^2}{(1+s)^2} u(x,s) \Phi(x,s) \, dx \, ds = \int_0^t \int_{\mathbb{R}^n} |\partial_t u(x,s)|^p  \Phi \, dx \, ds,
\end{aligned}
\end{equation}
together with the conditions $u(x,0)=\varepsilon  f(x)$  being satisfied in $H^1(\mathbb{R}^n)$.
\end{definition}
Integrating by parts in \eqref{energysol2}, one may derive the following equivalent formulation, which is often more convenient for our analysis:
\begin{equation} \label{energysol2-bis}
\begin{aligned}
&\int_{\mathbb{R}^n} \left[ \partial_t u(x,t) \Phi(x,t) - u(x,t) \partial_t \Phi(x,t) + \frac{\mu}{1+t} u(x,t) \Phi(x,t) \right] dx \\
&\quad + \int_0^t \int_{\mathbb{R}^n} u(x,s) \left[ \partial_s^2\Phi - \Delta \Phi - \frac{\partial}{\partial s} \left( \frac{\mu}{1+s} \Phi \right) + \frac{\nu^2}{(1+s)^2} \Phi \right] dx \, ds \\
&\quad = \int_0^t \int_{\mathbb{R}^n} |\partial_t u(x,s)|^p  \Phi \, dx \, ds + \varepsilon \int_{\mathbb{R}^n} \left[ \left( \mu f(x) + g(x) \right) \Phi(x,0) - f(x) \partial_t \Phi(x,0) \right] dx.
\end{aligned}
\end{equation}
In the following, we state the  main result of this article.
\begin{theorem} \label{blowup01}
Let $\mu > 0$, $d > \mu$, $\nu^2 > 0$, and $p > 1$. Assume that $f \in H^1(\mathbb{R}^n)$ and $g \in L^2(\mathbb{R}^n)$ are non-negative, compactly supported on $B(0, R)$, and do not vanish identically. 
Let $u$ be an energy solution of \eqref{Equ} on $[0, T_\varepsilon)$ with $\mathrm{supp}(u) \subset \{(x, t) : |x| \le t + R\}$. 
Then, there exists a constant $\varepsilon_0 > 0$ such that for all $0 < \varepsilon \le \varepsilon_0$, the lifespan $T_\varepsilon$ satisfies the upper bound estimate
\begin{equation}\label{life2}
\d T_\varepsilon \leq C_d\  \varepsilon^{-\frac{2(p-1)}{2-(n+d-1)(p-1)}}, \qquad \qquad  \forall d\in (\mu, \frac{p+1}{p-1}-n),
\end{equation}
whenever the exponent $p$ lies in the range
\begin{equation*}
1 < p < p_{\mathrm{Gla}}(n+\mu).
\end{equation*}
Here, $C_d$ is a positive constant independent of $\varepsilon$.
\end{theorem}

\begin{remark} \label{Perturb1}
We observe that the parameter $d$ can be any real number satisfying $d\ \in (\mu, \frac{p+1}{p-1}-n)$.  As specified in   Theorem \ref{blowup01},  the upper bound for the lifespan $T_\varepsilon$ is strictly larger than the one established in \cite{Our3} for $\delta \ge 0$. This increase in the upper bound is a direct consequence of our analytical framework. Specifically, for large time, we treat the combined operator 
\begin{equation*}
\frac{\mu}{1+t}\partial_t u + \frac{\nu^2}{(1+t)^2}u
\end{equation*}
as a perturbation of the scale-invariant damping term $\frac{d}{1+t}\partial_t u$.
\end{remark}

\begin{remark} \label{Perturb2}
It is well-established that for $\delta > 0$, the blow-up region typically extends to the critical value $p=p_{\mathrm{Gla}}(n+\mu)$. However, since our current proofs utilize a perturbative method, the result in  Theorem \ref{blowup01} 
is restricted to the strictly subcritical range $1 < p < p_{\mathrm{Gla}}(n+\mu)$. Extending these blow-up results to the critical case $p = p_{\mathrm{Gla}}(n+\mu)$  would necessitate a more refined analysis, as the critical dynamics lie beyond the reach of the perturbative approach employed here.
\end{remark}

%
\begin{remark}\label{rem-supp}
Since $f$ and $g$ are supported on $B_{\R^n}(0,R)$, one can see that $\mbox{\rm supp}(u) \ \subset\{(x,t)\in\R^n\times[0,\infty): |x|\le t+R\}$. Consequently, one  can choose any test function $\Phi$  which is not necessarily compactly supported.
\end{remark}

%

\section{Some auxiliary results }\label{sec3}
In this subsection, we construct a family of test functions designed to establish the blow-up results.
\subsection{Construction of a family of test functions}
Let us first define the conjugate operator $\mathcal{L^*_{\mu,\nu}}$ associated with the operator  $\mathcal{L_{\mu,\nu}}$ defined in  \eqref{Equl} as:
\begin{equation}\label{cp}
\mathcal{L}^*_{\mu,\nu} \psi := \partial^2_t \psi - \Delta \psi - \frac{\partial}{\partial t} \left( \frac{\mu}{1+t} \psi \right) + \frac{\nu^2}{(1+t)^2} \psi.\end{equation}

Following the strategies in \cite{Our3, Palmieri1, Tu-Lin}, traditional approaches for the case $\delta \ge 0$ seek a solution to the equality  
\begin{equation}\label{equal}
 \mathcal{L}^*_{\mu,\nu} U^{\eta} = 0.
  \end{equation}
  By separation of variables, we define the positive test function $U^{\eta}(x,t) := \xi^{\eta}(t)\phi^{\eta}(x)$. The spatial component $\phi^{\eta}(x)$ is given by
 \begin{equation}
\label{test12}
\phi^{\eta}(x) := 
\begin{cases} 
\displaystyle \int_{\mathbb{S}^{n-1}} e^{\eta x \cdot \omega} d\omega & \text{if } n \ge 2, \vspace{0.1cm} \\ 
e^{\eta x} + e^{-\eta x} & \text{if } n = 1.
\end{cases}
\end{equation}
As introduced in \cite{YZ06}, the function $\phi^{\eta}$ satisfies the identity $\Delta \phi^{\eta} = \eta^2 \phi^{\eta}$. 
The time-dependent component $\xi^{\eta}(t)$, as discussed in \cite{Palmieri1,Tu-Lin}, is the solution of the following ordinary differential equation:
\begin{equation}\label{lambda}
\frac{d^2 \xi^{\eta}(t)}{dt^2}-\eta^2\xi^{\eta}(t)-\frac{d}{dt}\left(\frac{\mu}{1+t}\xi^{\eta}(t)\right)+\frac{\nu^2}{(1+t)^2}\xi^{\eta}(t)=0.
\end{equation}
For the case $\delta \ge 0$, the temporal component $\xi^{\eta}(t)$ is typically expressed via modified Bessel functions \cite{Palmieri1, Tu-Lin}:
\begin{equation*}\label{lmabdaK}
\xi^{\eta}(t)=(\eta t+\eta)^{\frac{\mu+1}{2}}K_{\frac{\sqrt{\delta}}{2}}(\eta(t+1)),
\end{equation*}
where $K_{\alpha}(t)$ denotes the modified Bessel function of the second kind, defined as
\begin{equation*}
K_{\alpha}(t)=\int_0^\infty\exp(-t\cosh \zeta)\cosh(\alpha \zeta)d\zeta, \quad \alpha \in \mathbb{R}.
\end{equation*}

\par

Obviously, this classical construction breaks down when $\delta < 0$.  To overcome the specific difficulties associated with the negative discriminant case, 
we propose a construction that remains valid regardless of the sign of $\delta$.
Instead of requiring a test function that satisfies the identity \eqref{equal}, the central strategy of our proof involves a relaxation of the constraints. Specifically, we construct a test function such that  the left-hand side of \eqref{equal} remains non-positive.
This relaxation facilitates the construction of a family of explicit test functions. Specifically, we define the positive function 
 $\psi_d^{\eta}(x,t)$ via the separation of variables as follows:
\begin{equation}\label{test11}
\psi_d^{\eta}(x,t) := \rho_d^{\eta}(t)\phi^{\eta}(x), \quad \eta > 0,\quad d>0,
\end{equation}
where the spatial component $\phi^{\eta}(x)$ is given by
 \eqref{test12}, and 
 the temporal factor is defined by 
\begin{equation}\label{lmabdaK}
\rho_d^{\eta}(t) = (1+t)^{\frac{d}{2}} e^{-\eta t}.
\end{equation}
The properties of this construction are summarized in the following Lemma.
\begin{lemma}[Construction of a Family of Test Functions]\label{lemtestf}
Let $d > \mu$ and let $\psi_d^{\eta}(x,t)$ be the test function defined in \eqref{test11}. 
Then $\psi_d^{\eta}$ satisfies the adjoint differential identity
\begin{equation}\label{r01}
\mathcal{L}^*_{\mu,\nu} \psi_d^{\eta} = \mathcal{K}_d^{\eta}(t)\psi_d^{\eta},
\end{equation}
 where  $\mathcal{L}^*_{\mu,\nu}$ is defined in \eqref{cp}, and 
 \begin{equation}\label{Kb}
 \mathcal{K}_d^{\eta}(t) := \frac{4\nu^2 + (d-2\mu)(d-2)}{4(1+t)^2} - \frac{\eta(d-\mu)}{1+t}.
 \end{equation}
Furthermore, there exists $\tilde{\eta} = \tilde{\eta}(d, \mu, \nu) \ge 2$ defined by
\begin{equation}\label{tildeeta}\tilde{\eta} := \max(2,\frac{4\nu^2 + (d-2\mu)(d-2)}{4(d-\mu)}),
\end{equation}
 such that 
\begin{equation}\label{neg}
\mathcal{K}_d^{\eta}(t) \le 0, \quad \forall  t \ge 0, \quad \text{for} \quad \eta \ge \tilde{\eta}.
\end{equation}
 \end{lemma}

 \begin{proof}
 By substituting the separation of variables ansatz $\psi_d^{\eta}(x,t) = \rho_d^{\eta}(t)\phi^{\eta}(x)$ into the operator \eqref{cp} and utilizing the eigenvalue relation $\Delta \phi^{\eta} = \eta^2 \phi^{\eta}$, we obtain
 \begin{equation}\label{r1}
 \mathcal{L}^*_{\mu,\nu} \psi_d^{\eta} = \left[ \frac{d^2 \rho^{\eta}_d}{dt^2} - \frac{d}{dt} \left( \frac{\mu}{1+t} \rho^{\eta}_d \right) + \left( \frac{\nu^2}{(1+t)^2} - \eta^2 \right) \rho^{\eta}_d \right] \phi^{\eta}(x).
 \end{equation}
 A direct calculation for $\rho_d^{\eta}(t) = (1+t)^{\frac{d}{2}} e^{-\eta t}$ shows that
 \begin{equation}\label{r2}
   \frac{d^2 \rho^{\eta}_d}{dt^2} - \frac{d}{dt} \left( \frac{\mu}{1+t} \rho^{\eta}_d \right) + \left( \frac{\nu^2}{(1+t)^2} - \eta^2 \right) \rho^{\eta}_d =\mathcal{K}_d^{\eta}(t) \rho^{\eta}_d(t),
 \end{equation}
  where
$ \mathcal{K}_d^{\eta}(t)$ is given by \eqref{Kb}. Clearly, by combining \eqref{test11}, \eqref{r1} and \eqref{r2}, we deduce 
\eqref{r01}. 
 
For $d > \mu$, we ensure the  negativity of $\mathcal{K}_d^{\eta}(t)$ by choosing a sufficiently large parameter $\eta$. 
Specifically, we set the threshold
\begin{equation}\label{tildeeta0}
\eta \ge \tilde{\eta} :=\max(2, \frac{4\nu^2 + (d-2\mu)(d-2)}{4(d-\mu)}),
\end{equation}
which guarantees the inequality at the initial time $t=0$.
 By the monotonicity of the terms in $t$, this condition remains valid for all $t \ge 0$, namely:
\begin{equation}\label{neg}
\mathcal{K}_d^{\eta}(t) \le 0, \quad \forall  t \ge 0, \quad \text{for} \quad \eta \ge \tilde{\eta}.
\end{equation}
This end the proof of Lemma \ref{lemtestf}.
 \end{proof}
\begin{remark}
The principal merit of this construction lies in its independence from the sign of $\delta$. Specifically, for sufficiently large $\eta$, the explicit function $\psi^{\eta}_d$ serves as a  subsolution to the operator defined  in \eqref{cp}.  Crucially, relaxing the requirement from a strict equality to a differential inequality does not affect the subsequent steps of the blow-up argument.
\end{remark}

With the family of test functions constructed in Lemma \ref{lemtestf} at hand, we adapt the strategy developed in \cite{Our3} to the present setting. By exploiting the structural properties of this family, we derive the necessary differential inequalities to establish the blow-up result. To this end, for $\eta > 0$ and $d > \mu$, we define the following  functional:
 \begin{equation}\label{wF1def}
\mathcal{F}_d^{\eta}(t) := 
 \int_{\mathbb{R}^n} u(x, t) \psi_d^{\eta}(x,t) dx.
 \end{equation}
The objective of this subsection is to derive the necessary lower estimates for $\mathcal{F}_d^{\eta}(t)$.

\subsection{Lower estimates for the functional $\U(t)$}\label{sec32}
In this subsection, we establish two sharp lower bounds for the functional $\mathcal{F}_d^{\eta}(t)$, as stated in the following lemma
\begin{lemma}
\label{F1-n}
Assume that the assumptions in Theorem \ref{blowup01} hold. Let $u$ be an energy solution of \eqref{energysol2}. Then,  for all $d>\mu$, there exists $ \eta_0(d,\mu,\nu)\ge 2$ such that  we have
\begin{equation}
\label{F1postive0}
\U(t)>0\
\quad\text{for all}\ t \in [0,T), \quad\text{for all}\ \eta\ge\eta_0,
\end{equation}
and
\begin{equation}
\label{F1postive}
\U(t)\ge \frac{
C^{\eta}_0(f,g)}{4\eta}\, \e, 
\quad\text{for all}\ t \in [1,T),\quad\text{for all}\ \eta\ge\eta_0.
\end{equation}
Here $ \eta_0:=max (d+2,\frac{4\nu^2 + (d-2\mu)(d-2)}{4(d-\mu)})$. Furthermore, the constant $C^{\eta}_0(f,g)$ is given by
\begin{equation}\label{C0fg}
C^{\eta}_0(f,g):= \int_{\R^n} \big(f(x) +g(x)\big)\phi^{\eta}(x)dx.
\end{equation}
\end{lemma}
\begin{proof} 

Let $d > \mu$, $\eta > 0$, and $t \in [0, T)$. We consider the test function $\psi_{d}^{\eta}$ defined by the separation of variables
\begin{equation}\label{lk}\psi_{d}^{\eta}(x, t) = \rho_{d}^{\eta }(t) \phi^{\eta}(x),
\end{equation}
where the temporal component $\rho_{d}^{\eta }(t)$ and the spatial component $\phi^{\eta}(x)$ are as defined in \eqref{lmabdaK} and \eqref{test12}, respectively. By substituting $\psi_{d}^{\eta}$ into the weak formulation \eqref{energysol2-bis} and utilizing the identity \eqref{r01}, we obtain the following integral equality:
\begin{equation} \label{e0}
\begin{aligned}
&\int_{\mathbb{R}^n} \left[ \partial_t u(x,t) \psi_d^{\eta}(x,t)- u(x,t) \partial_t \psi_d^{\eta}(x,t)+ \frac{\mu}{1+t} u(x,t) \psi_d^{\eta}(x,t)\right] dx \\
&\quad + \int_0^t\mathcal{K}_{d}(s) \int_{\mathbb{R}^n} u(x,s) \psi_d^{\eta}(x,s)dx \, ds = 
\int_{0}^{t} \int_{\mathbb{R}^n} |\partial_tu(x,s)|^p \psi^{\eta}_{d}(x,s) \, dx \, ds\\
&\quad  + \varepsilon \int_{\mathbb{R}^n} \left[ \left( \mu f(x) + g(x) \right) \psi_d^{\eta}(x,0) - f(x) \partial_t \psi_d^{\eta}(x,0) \right] dx.
\end{aligned}
\end{equation}
It is worth noting that, although the test function $\psi_{d}^{\eta}$ lacks compact support with respect to the spatial variable, the validity of \eqref{e0} is ensured by the fact that the solution $u(\cdot, t)$ has compact support in $\mathbb{R}^n$ for every $t \in [0, T)$. More precisely, this can be rigorously justified by replacing $\psi_{d}^{\eta}$ with the localized test function $\psi_{d}^{\eta} \chi$, where $\chi \in C_0^\infty(\mathbb{R}^n)$ denotes a smooth cut-off function satisfying $\chi \equiv 1$ on $\text{supp}(u)$.\\
From the definition of $\psi_{d}^{\eta}(x, t)$ in \eqref{lk} and \eqref{lmabdaK}, it follows that:
\begin{equation}\label{01}
{\partial_ t} \psi^{\eta}_d(x,t) =  \frac{d}{2(1+t)} \psi^{\eta}_d(x,t)- \eta \psi^{\eta}_d(x,t).
\end{equation}
Evaluating this at $t=0$, and exploiting the fact $\rho^{\eta}(0)=1$, we have
\begin{equation}\label{02}
{\partial_ t} \psi^{\eta}_{d}(x,0) =  (\frac{d}{2} - \eta) \psi^{\eta}_{d}(x,0)
=(\frac{d}{2} - \eta) \phi^{\eta}(x).
\end{equation}
Substituting \eqref{01} and \eqref{02} into \eqref{e0}, we arrive at the following identity
\begin{equation}
\label{eq09}
\begin{aligned}
\int_{\mathbb{R}^n} \left[ \partial_tu(x,t) + \eta u(x,t) +\frac{2\mu-d}{2(1+t)}u(x,t)\right] \w(x,t) \, dx\qquad\qquad\qquad\qquad\qquad\qquad\qquad\\
 + \int_0^t \mathcal{K}_{d}(s)  \int_{\mathbb{R}^n} u(x,s) \w(x,s) \, dx \, ds 
= \int_{0}^{t} \int_{\mathbb{R}^n} |\partial_tu(x,s)|^p \psi^{\eta}_{d}(x,s) \, dx \, ds
+\varepsilon C^{\eta}_1(f,g),
\end{aligned}
\end{equation}
where 
\begin{equation}\label{Cfg}
C^{\eta}_1(f,g)= (\eta+\frac{2\mu-d}{2})\int_{\R^n} f(x)\phi^{\eta}(x)dx +
\int_{\R^n}g(x)\phi^{\eta}(x)dx.
\end{equation}
 In view of the definition of $\mathcal{F}_d^{\eta}(t)$ in \eqref{wF1def} and the identity \eqref{01}, a direct calculation yields
\begin{equation}\label{deriv}
\int_{\mathbb{R}^n}  \partial_tu(x,t)  \w(x,t) \, dx
=\frac{d}{dt} \U (t)+\Big(\eta-\frac{d}{2(1+t)}\Big)\U(t).
\end{equation}
By substituting \eqref{deriv} into \eqref{eq09}, we obtain the following differential-integral identity:
\begin{equation}
\begin{array}{l}\label{eq6}
\d \frac{d}{dt}\U(t)+
\gamma(t)\U(t)
+ \int_0^t \mathcal{K}_{d}(s)  \U(s) \, ds  
= \int_{0}^{t} \int_{\mathbb{R}^n} |\partial_tu(x,s)|^p \psi^{\eta}_{d}(x,s) \, dx \, ds
+\varepsilon C^{\eta}_1(f,g),
\end{array}
\end{equation}
where the time-dependent coefficient $\gamma(t)$ is defined by
\begin{equation}\label{gamma}
\gamma(t) := 2\eta + \frac{\mu - d}{1+t}.
\end{equation}
Taking into account the non-negativity of the nonlinear term, we multiply \eqref{eq6} by the integrating factor $\Gamma(t):=e^{2\eta t}(1+t)^{\mu-d}$ and integrate the resulting inequality over $[0,t]$ to obtain
\begin{equation}\label{wwF1+}
\d \U(t)+\frac1{\Gamma(t)}\int_0^t\Gamma(s)\int_{0}^{s}  \mathcal{K}_{d}(\tau)\U(\tau) d\tau ds
\ge \frac{\U(0)}{\Gamma(t)}+\frac{{\e}C^{\eta}_1(f,g)}{\Gamma(t)}\int_0^t  \Gamma(s)ds.
\end{equation}
Observing that $\d \U(0)=\ep \int_{\R^n}f(x) \phi^{\eta}(x)dx>0$, we further note if  $\eta  \ge d + 2$, it follows that:
\begin{equation}\label{Cfg0}
C^{\eta}_1(f,g)\ge \int_{\R^n} \big(f(x) +g(x)\big)\phi^{\eta}(x)dx=C^{\eta}_0(f,g)>0.
\end{equation}
Consequently, by virtue of \eqref{wwF1+}, it follows that:
\begin{equation}\label{wF1+}
\d \U(t)+\frac1{\Gamma(t)}\int_0^t\Gamma(s)\int_{0}^{s}  \mathcal{K}_{d}(\tau)\U(\tau) d\tau ds>0, \qquad \forall\  t \in [0, T), \qquad \eta  \ge d  + 2.
\end{equation}
By Lemma \ref{lemtestf}, for any $\eta \ge \tilde{\eta}$,  we have
\begin{equation}\label{negb}
\mathcal{K}_d^{\eta}(t) \le 0, \quad \forall  t \ge 0,
\end{equation}
where  $\tilde{\eta}$ is given by  \eqref{tildeeta}.
By invoking \eqref{wF1+} and the fact that $\U(0) > 0$, we deduce that:
\begin{equation}\label{pos}
\U(t) > 0, \quad \forall  t \in [0,T), \qquad \eta  \ge \eta_0,
\end{equation}
where $\eta_0 := \max(d+2,\frac{4\nu^2 + (d-2\mu)(d-2)}{4(d-\mu)})$.
Indeed, given that $\U(0) > 0$,   and the mapping $t \mapsto \U(t)$ is continuous, let us assume for the sake of contradiction that there exists a first time $t_0 \in (0, T)$ such that $\U(t_0) = 0$. However, applying a comparison argument based on \eqref{wF1+}, and exploiting \eqref{negb} yields a contradiction at $t_0$. 
Thus, $\U(t)$ remains strictly positive for all $t \in [0, T)$, and hence \eqref{F1postive0} holds.\\

Now, substituting \eqref{Cfg0}, \eqref{F1postive0} and \eqref{negb} into the identity \eqref{eq6}, we obtain:
\begin{equation}
\label{eq6b}
\d \frac{d}{dt}\U(t)+\gamma(t)\U(t)\ge \int_{0}^{t} \int_{\mathbb{R}^n} |\partial_tu(x,s)|^p \psi^{\eta}_{d}(x,s) \, dx \, ds
 +\e \, C^{\eta}_0(f,g), >0, \quad \forall\  t \in [0, T), \quad \eta  \ge \eta_0.
\end{equation}
Since the nonlinear tem  is non-negative, we multiply the inequality \eqref{eq6b} by the integrating factor $\Gamma(t) = e^{2\eta t}(1+t)^{\mu-d}$ and integrate the result over the interval $[0, t]$. This leads to:
\begin{equation}\label{wF1++}
\d \U(t)
\ge \frac{\U(0)}{\Gamma(t)}+\frac{{\e}C^{\eta}_0(f,g)}{\Gamma(t)}\int_0^t  \Gamma(s)ds\quad \forall\  t \in [0, T), \quad \eta  \ge \eta_0.
\end{equation}
Observing that $\d \U(0)=\ep \int_{\R^n}f(x) \phi^{\eta}(x)dx>0$, the estimate \eqref{wF1++} yields
\begin{align}\label{est-G1-1}
\U(t)
\ge {\e}C^{\eta}_0(f,g)(1+t)^{d-\mu}e^{-2\eta t}\int_{t/2}^t(1+t)^{\mu-d}e^{2\eta s}ds\quad \forall\  t \in [0, T), \quad \eta  \ge \eta_0.
\end{align}
Consequently, we obtain
\begin{align}\label{est-G1-2}
 \U(t)
\ge \e C^{\eta}_0(f,g)e^{-2\eta t}\int^t_{t/2}e^{2\eta s}ds\ge \frac{\e C^{\eta}_0(f,g)}{2\eta}(1-e^{-\eta t}), \quad \forall\  t \in [0, T), \quad \eta  \ge \eta_0.
\end{align}
Finally, taking into account that
\begin{equation}
e^{-\eta t}\le \frac1{\eta}\le \frac12, 
\quad \forall\  t \in [1, T), \quad \eta  \ge \eta_0,
\end{equation}
 we obtain \eqref{F1postive}. This ends the proof of Lemma \ref{F1-n}.
\end{proof}
\begin{remark}
It is worth noting that the proof of Lemma \ref{F1-n} is independent of the specific power-type structure of the nonlinearities, provided they remain non-negative. Since the argument relies on the asymptotic behavior of the linear operator, the result can be extended to any nonlinearity $F(u, \partial_t u) \ge 0$.
\end{remark}

By virtue of the estimates established in Lemma \ref{F1-n}, we now derive the necessary lower bounds for the functional
\begin{equation}\label{wF2def}
\V (t) := \int_{\mathbb{R}^n} \partial_t u(x, t) \psi_d^{\eta}(x,t) dx.
\end{equation}
These estimates ensure the coercive growth required for the subsequent blow-up analysis.

\subsection{Lower estimates for the functional $\V(t)$}\label{sec33}
In this subsection, we establish two sharp lower bounds for the functional $\V(t)$, as stated in the following lemma.
\begin{lemma}
\label{F21}
Assume that the assumptions in Theorem \ref{blowup01} hold. Let $u$ be an energy solution of \eqref{energysol2}. Then,  for all $d>\mu$, there exists $ \eta_1(d,\mu,\nu)\ge 2$ such that  we have
\begin{equation}
\label{F2postive0}
\V(t)\ge 0, 
\quad\text{for all}\ t \in [0,T)\quad\text{for all}\ \eta\ge\eta_1,
\end{equation}
and
\begin{equation}
\label{F2postive}
\V(t)\ge \frac{
C^{\eta}_0(f,g)}{18}\, \e, 
\quad\text{for all}\ t \in [1,T)\quad\text{for all}\ \eta\ge\eta_1,
\end{equation}
where
 \begin{equation}\label{eta1}
\eta_1:=max (2d+2\mu+2\nu+2,\frac{4\nu^2 + (d-2\mu)(d-2)}{4(d-\mu)}),
\end{equation} 
and the constant $C^{\eta}_0(f,g)$ is given by
\begin{equation}\label{C0fg2}
C^{\eta}_0(f,g)= \int_{\R^n} \big(f(x) +g(x)\big)\phi^{\eta}(x)dx.
\end{equation}
\end{lemma}
\begin{proof}
Let $d > \mu$, $\eta > \eta_1$, and $t \in [0, T)$. 
 we consider the functionals  $\U(t)$ and $\V(t)$ as defined  in \eqref{wF1def} and \eqref{wF2def}, respectively.
 In view of \eqref{01}, we obtain the identity
  \begin{equation}\label{deriv2}
\V(t) = \frac{d \U(t)}{dt} + \left( \eta - \frac{d}{2(1+t)} \right) \U(t).
\end{equation}
Substituting \eqref{deriv2} into \eqref{eq6}, we conclude
\begin{equation}
\begin{array}{l}\label{eq5bis}
\d \V(t)+
\gamma_1(t)\U(t)
+ \int_0^t \mathcal{K}_{d}(s)  \U(s) \, ds  
= \int_{0}^{t} \int_{\mathbb{R}^n} |\partial_tu(x,s)|^p \psi^{\eta}_{d}(x,s) \, dx \, ds
+\varepsilon C^{\eta}_1(f,g),
\end{array}
\end{equation}
 where the time-dependent coefficient $\gamma(t)$ is defined by
\begin{equation}\label{gamma1}
\gamma_1(t) := \eta + \frac{2\mu - d}{2(1+t)}.
\end{equation}
Now, taking the time-derivative of the  equation \eqref{eq5bis}, we infer that
\begin{equation}\label{F1+bis}
\d \frac{d \V}{dt}(t)+
\gamma_1(t)
\frac{d \U}{dt}(t)-\frac{2\mu-d}{2(1+t)^2}\U(t) +  \mathcal{K}_{d}(t)  \U(t)  
=\int_{\mathbb{R}^n} |\partial_tu(x,t)|^p \psi^{\eta}_{d}(x,t) \, dx. 
\end{equation}
Employing     \eqref{deriv2}, the equation  \eqref{F1+bis} yields
\begin{equation}\label{F1+bisb}
\d \frac{d \V}{dt}(t)+
\gamma_1(t) \V(t) 
\d =\int_{\mathbb{R}^n} |\partial_tu(x,t)|^p \psi^{\eta}_{d}(x,t) \, dx
+\left[(\eta-\frac{d}{2(1+t)})
\gamma_1(t)
+\frac{2\mu-d}{2(1+t)^2}-\mathcal{K}_{d}(t)\right]\U(t).
\end{equation}
By combining the definitions of $\gamma(t)$ and $\gamma_1(t)$ from \eqref{gamma} and \eqref{gamma1} with the identity \eqref{F1+bisb}, we conclude that
\begin{equation}\label{G2+bis3}
\d \frac{d \V}{dt}(t)+\frac{3
\gamma(t)}{4}\V(t)=
\int_{\mathbb{R}^n} |\partial_tu(x,t)|^p \psi^{\eta}_{d}(x,t) \, dx
+\lambda(\eta,t)  \U(t)+\Sigma(t),
\end{equation}
where 
\begin{equation}\label{sigma1-exp}
\begin{array}{rl}
\d \lambda(\eta,t) :=&\frac{\eta^2}2+\frac{2d-\mu}{4(1+t)}\eta-\frac{\nu^2}{(1+t)^2}+\frac{(d-2\mu)(3d-\mu-8)}{
8(1+t)^2},\\
\Sigma(t):=&\d \Big(\frac{\eta}{2}-\frac{\mu+d}{4(1+t)}\Big)\Big(\V(t)+\gamma_1(t)\U(t)\Big).
\end{array} 
\end{equation}
 Using the fact that $d > \frac{\mu}{2}$, it follows that for all $t \ge 0$:
\begin{equation}
\lambda(\eta,t) \ge \frac{\eta^2}{2} -  \nu^2-d\ge\frac{\eta^2}{2} -(d+\nu+1)^2 .
\end{equation}
Consequently, by invoking the positivity of $\U(t)$ from \eqref{F1postive0} and choosing $\tilde{\eta}_0 = \max(\eta_0, 2(d + \nu+1))$, we ensure that $\lambda(\eta,t) \ge 0$. We thus conclude that:
\begin{equation}\label{sigma2}
\lambda(\eta,t)  \U(t) \ge 0, \quad \forall \ t  \in [0,T), \quad \eta\ge \tilde \eta_0. 
\end{equation}
Moreover,
by substituting \eqref{negb}, \eqref{F1postive0}, and \eqref{Cfg0} into \eqref{eq5bis}, we obtain
 \begin{equation}
\label{v01}
\V(t)+
\gamma_1(t)\U(t)
\ge  \int_{0}^{t} \int_{\mathbb{R}^n} |\partial_tu(x,s)|^p \psi^{\eta}_{d}(x,s) \, dx \, ds+
\varepsilon C^{\eta}_0(f,g),\quad \forall \ t  \in [0,T), \quad \eta\ge \tilde \eta_0.
\end{equation}
Therefore, by choosing  $\eta\ge \eta_1:= \max (\tilde \eta_0,d+\mu)$, we have 
\begin{equation}\label{sigma1}
\d \Sigma(t) \ge \frac{\eta}{4}  \int_{0}^{t} \int_{\mathbb{R}^n} |\partial_tu(x,s)|^p \psi^{\eta}_{d}(x,s) \, dx \, ds+ \frac{\eta \e}{4} \, C^{\eta}_0(f,g), \quad  \quad \forall \ t  \in [0,T), \quad \eta\ge  \eta_1. 
\end{equation}
Gathering all the above results, namely \eqref{G2+bis3}, \eqref{sigma2} and \eqref{sigma1}, we end up with the following estimate
\begin{eqnarray}\label{G2+bis41}
\d \frac{d \V}{dt}(t)+\frac{3
\gamma(t)}{4}\V(t) \ge \d \frac{\eta \e}{4} \, C^{\eta}_0(f,g)+  
+\frac{\eta}{4}
\  \int_{0}^{t} \int_{\mathbb{R}^n} |\partial_tu(x,s)|^p \psi^{\eta}_{d}(x,s) \, dx \, ds\nonumber\\
+\int_{\mathbb{R}^n} |\partial_tu(x,t)|^p \psi^{\eta}_{d}(x,t) \, dx, \quad  \quad \forall \ t  \in [0,T), \quad \eta\ge  \eta_1.
\end{eqnarray}
At this level, we can eliminate  the nonlinear terms\footnote{ In fact, for a subsequent use in the proof of the main result, we choose here to keep the nonlinear terms up to this step in our computations. Otherwise,  omitting  the nonlinear terms can be done earlier in the proof of this lemma.} and we write 
\begin{equation}\label{G2+bis4}
\begin{array}{rcl}
\d \frac{d \V
}{dt}(t)+\frac{3
\gamma(t)}{4}\V(t) &\ge& \d \frac{\eta\e }{4} \, C^{\eta}_0(f,g), \quad  \quad \forall \ t  \in [0,T), \quad \eta\ge  \eta_1. 
\end{array}
\end{equation}
Integrating \eqref{G2+bis4} over $(0,t)$ after multiplication by  
$(1+t)^{\frac{3(\mu-d)}{4}}e^{\frac{3\eta t}2 }$, we obtain
\begin{eqnarray}\label{fgr}
 \V(t)
\ge 
\frac{\eta \e}4C^{\eta}_0(f,g)
(1+t)^{\frac34(d-\mu)}e^{-\frac32 \eta t}
\int_{0}^t
(1+s)^{\frac34(\mu-d)}e^{\frac32 \eta s}ds\nonumber\\
+\V(0)(1+t)^{\frac34(d-\mu)}e^{-\frac32 \eta t}\quad  \quad \forall \ t  \in [0,T), \quad \eta\ge  \eta_1.
\end{eqnarray}
Observing that $\d \V(0)=\ep \int_{\R^n}g(x) \phi^{\eta}(x)dx>0$, the estimate \eqref{fgr} yields
\begin{equation}\label{ghr1}
\V(t)\ge 
\frac{\eta \e}4C^{\eta}_0(f,g)(1+t)^{\frac34(d-\mu)}e^{-\frac32 \eta t}
\int_{0}^t(1+s)^{\frac34(\mu-d)}e^{\frac32 \eta s}ds\quad  \quad \forall \ t  \in [0,T), \quad \eta\ge  \eta_1,
\end{equation}
and hence \eqref{F2postive0} follows. Furthermore, in view of \eqref{ghr1}, we obtain
\begin{equation}\label{ghr2}
\V(t)\ge 
\frac{\eta \e}4C^{\eta}_0(f,g)e^{-\frac32 \eta t}
\int_{t/2}^te^{\frac32 \eta s}ds\ge \frac{\e C^{\eta}_0(f,g)}{6}(1-e^{-\frac{3\eta t}4}),\quad  \quad \forall \ t  \in [0,T), \quad \eta\ge  \eta_1.
\end{equation}
Finally, taking into account that
\begin{equation}\label{bb}
e^{-\frac{3\eta t}4}\le \frac4{3\eta}\le \frac23, 
\quad \forall\  t \in [1, T), \quad \eta  \ge \eta_1.
\end{equation}
It follows from \eqref{ghr2} and \eqref{bb} that \eqref{F2postive} holds, which concludes the proof of Lemma \ref{F21}.
\end{proof}

\par

Before proving  Theorem \ref{blowup01}, 
we recall a technical estimate for $\phi^{\eta}$ that will be essential for our analysis.
\begin{lemma}[\cite{YZ06}] \label{lem1}
Let $r > 1$. There exists a constant $ C(n, R, \eta, r) > 0$ such that, for all $t \ge 0$,
\begin{equation}
\label{psi}
\int_{|x| \le t+R} \left( \phi^{\eta}(x) \right)^r dx \le C(n,R,\eta,r) e^{r \eta t} (1+t)^{\frac{(n-1)(2-r)}{2}}.
\end{equation}
\end{lemma}

\section{Proof of Theorem \ref{blowup01}}\label{sec44}
This section is devoted to the proof of Theorem \ref{blowup01}. By utilizing the coercive property of $\V(t)$ established in Lemma \ref{F21}, we shall prove the blow-up result for the solutions to \eqref{Equ}.

Let $d > \mu$ and assume that the parameter $\eta$ satisfies $\eta \ge \eta_1$, where $\eta_1$ is  prescribed by \eqref{eta1}.
To study the evolution of the system, we first define the following functional: 
\begin{equation}\label{L1}
 L^{\eta}_d(t):=
\frac{1}{8}\int_{1}^t  \int_{\R^n}|\partial_tu(x,s)|^p\w(x,s)dx ds
+\frac{C_0^{\eta}(f,g) \e}{24}.
\end{equation}
Furthermore, we introduce
$$\mathcal{H}_d^{\eta}(t):= \V(t)-L^{\eta}_d(t).$$
Thanks to \eqref{G2+bis41},  we see that $\mathcal{H}^{\eta}_d$
 satisfies
\begin{equation}\label{G2+bis6}
\begin{array}{rcl}
\d \frac{d}{dt}\mathcal{H}_d^{\eta}(t)+\frac{3
\gamma(t)}{4}\mathcal{H}_d^{\eta}(t) &\ge& \d \left(\frac{\eta}{4}-\frac{3
\gamma(t)}{32}\right)\int_{1}^t \int_{\R^n}|\partial_tu(x,s)|^p\w(x,s)dx ds\vspace{.2cm}\\ &&+  \d \frac{7}{8}\int_{\R^n}|\partial_tu(x,t)|^p\w(x,t) dx\\
&&+\d \frac{ C^{\eta}_0(f,g)\e}{4} \big( \eta-\frac{3
\gamma(t)}{24} \big), \quad \forall \ t \in[ 1,T).
\end{array}
\end{equation}
From \eqref{gamma}, since $d \ge \mu$ and $\eta \ge \eta_1\ge  d $, it follows that $\eta \le \gamma(t) \le 2\eta$. Consequently, we have:
$$\frac{\eta}{4} - \frac{3\gamma(t)}{32} \ge \frac{\eta}{4} - \frac{6\eta}{32} = \frac{\eta}{16} \ge 0.$$
Furthermore, by applying the upper bound $\gamma(t) \le 2\eta$, we obtain:
$$\eta  - \frac{3\gamma(t)}{24}  \ge \eta - \frac{6\eta}{24}=\frac{3\eta}4  \ge 0.$$
Thanks to the above, we write
\begin{equation}\label{GG}
\d \frac{d}{dt}\mathcal{H}_d^{\eta}(t)+\frac{3
\gamma(t)}{4}\mathcal{H}_d^{\eta}(t) \ge0, \quad \forall \ t \in [1,T).
\end{equation}
Integrating \eqref{GG} over $(1,t)$ after multiplication by  
$(1+t)^{\frac{3(\mu-d)}{4}}e^{\frac{3\eta t}2}$, we obtain
\begin{align}\label{est-G111}
 \mathcal{H}_d^{\eta}(t)
\ge \mathcal{H}_d^{\eta}(1)
 2^{\frac34(\mu-d)}e^{\frac{3 \eta}2 }(1+t)^{\frac34(d-\mu)}e^{-\frac32 \eta t}, \quad \forall \ t \in [1,T).
\end{align}
Using Lemma \ref{F21} , one can see that $\d \mathcal{H}^{\eta}_d(1)=\V(1)-\frac{C_0^{\eta}(f,g) \e}{24} \ge \frac{C_0^{\eta}(f,g)}{18}\e -\frac{C_0^{\eta}(f,g) \e}{24}\ge 0$. \\
Consequently, we infer that
\begin{equation}
\label{45}
\GG(t)\geq  L^{\eta}_d(t), \quad \forall \ t \in [1,T).
\end{equation}
Employing the H\"{o}lder's inequality together with the estimates \eqref{psi} and \eqref{F2postive}, a lower bound for the nonlinear term can written as
\begin{equation}\label{vt09}
\begin{array}{rcl}
\d \int_{\R^n}|\partial_tu(x,t)|^p\w(x,t)dx &\geq&\d (\V(t))^p\left(\int_{|x|\leq t+R}\w(x,t)dx\right)^{1-p} \vspace{.2cm}\\ &\geq& C (\V(t))^p (1+t)^{-\frac{(d+n-1)(p-1)}2},
 \qquad \forall \ t \in [1,T).
\end{array}
\end{equation}
Now, recall the definition of $ L^{\eta}_d(t)$, given by \eqref{L1}, and injecting \eqref{45} in \eqref{vt09}, we conclude that
\begin{equation}
\label{bl}
\frac{d}{dt}L^{\eta}_d(t)\geq 
C (L_d^{\eta}(t))^p (1+t)^{-\frac{(d+n-1)(p-1)}2}, \qquad \forall \ t \in [1,T).
\end{equation}

Let us recall that the condition $1 < p < p_{\mathrm{Gla}}(n+\mu)$ ensures that $\frac{(n+\mu-1)(p-1)}{2} < 1$, which is equivalent to $\mu < \frac{p+1}{p-1} - n$. Consequently, for any choice of the parameter $d \in (\mu, \frac{p+1}{p-1} - n)$, the exponent 
\begin{equation}
\theta := 1 - \frac{(n+d-1)(p-1)}{2}
\end{equation}
is strictly positive. By virtue of the positivity of $L_d^{\eta}(t)$, we may divide both sides of the differential inequality \eqref{bl} by $(L_d^{\eta}(t))^p$. Integrating the resulting expression over the interval $[1, t]$ for any $t \in [1, T)$, we find
\begin{equation}\label{int6}
\frac{1}{p-1} \left[ \big(L_d^{\eta}(1)\big)^{-(p-1)} - \big(L_d^{\eta}(t)\big)^{-(p-1)} \right] \geq \frac{C}{\theta} \left[ (1+t)^{\theta} - 2^{\theta} \right].
\end{equation}
Since $\theta > 0$ and $L_d^{\eta}(t) > 0$ for all $t \in [1, T)$, we may neglect the second non-negative term on the left-hand side of \eqref{int6}. Recalling from \eqref{L1} that $L_d^{\eta}(1) = \frac{C_0^{\eta}(f,g) \varepsilon}{24}$, we arrive at the estimate
\begin{equation}\label{ep0}
 \varepsilon^{-(p-1)} \geq C \left[ (1+t)^{\theta} - 2^{\theta} \right],
\end{equation}
where $C > 0$ is a constant independent of $\varepsilon$. From \eqref{ep0}, it is clear that there exist $\varepsilon_0>0$  sufficiently small, such that if  $\varepsilon\le \varepsilon$,  the lifespan $T_{\varepsilon}$ must be finite and satisfies the upper bound
\begin{equation}
T_{\varepsilon} \leq  C\varepsilon^{-\frac{p-1}{\theta}}.
\end{equation}
This achieves the proof of Theorem \ref{blowup01}.\hfill $\Box$
\section*{Declarations}
\begin{itemize}
    \item \textbf{Conflict of interest:} The author declares no conflict of interest.
    \item \textbf{Data availability:} Our manuscript has no associated data.
\end{itemize}

\end{document}